\documentclass{amsart}

\usepackage{xypic}
\xyoption{curve}
\usepackage{amsrefs}
\usepackage{longtable,booktabs}
\usepackage{xcolor}
\usepackage[pdf]{pstricks}
\usepackage{pst-plot}
\usepackage{lscape}
\usepackage{amssymb}

\newtheorem{thm}{Theorem}[section]
\newtheorem{lemma}[thm]{Lemma}
\newtheorem{prop}[thm]{Proposition}

\theoremstyle{definition} 
\newtheorem{rmk}[thm]{Remark}
\newtheorem{step}[thm]{Step}

\newcommand{\C}{\mathbb{C}}	
\newcommand{\F}{\mathbb{F}} 
\newcommand{\Z}{\mathbb{Z}} 

\newcommand{\map}{\rightarrow}

\newcommand{\Xtwofour}{Y}
\newcommand{\Xtwofive}{Q}
\newcommand{\Xonefour}{C}
\newcommand{\Xzerofour}{Z}

\newcommand{\hsf}{\widetilde{2}}

\newcommand{\tmf}{\mathrm{tmf}}
\newcommand{\mmf}{\mathrm{mmf}}

\DeclareMathOperator{\Sq}{Sq}

\DeclareMathOperator{\colim}{colim}

\newcommand{\lemdeg}[1]{$({#1})$}

\newcommand{\D}{\Delta}
\newgray{gridline}{0.8}

\newrgbcolor{vzerocolor}{1 0 0}
\newrgbcolor{vonecolor}{0.1 0.1 0.7}
\newrgbcolor{vtwocolor}{0.1 0.8 0.1}
\newrgbcolor{vthreecolor}{0.5 0.5 0.5}

\newcommand{\kappabar}{\overline{\kappa}}

\begin{document}

\title[$\F_2$-synthetic methods]
{Classical Stable Homotopy Groups of Spheres via $\F_2$-Synthetic Methods}

\author{Robert Burklund}
\address{Department of Mathematical Sciences, University of Copenhagen, Denmark}
\email{rb@math.ku.dk}

\author{Daniel C.\ Isaksen}
\address{Department of Mathematics\\
Wayne State University\\
Detroit, MI 48202, USA}
\email{isaksen@wayne.edu}

\author{Zhouli Xu}
\address{Department of Mathematics, UC San Diego, La Jolla, CA 92093, USA}
\email{xuzhouli@ucsd.edu}

\thanks{
The first author was supported by NSF grant DMS-2202992, by the DNRF through the Copenhagen center for Geometry and Topology (DNRF151) and by Villum Fonden through the Young Investigator Program.
The second author was supported by NSF grant DMS-1904241.
The third author was partially supported by NSF grant DMS-2105462.}



\date{\today}

\begin{abstract}
We study the $\F_2$-synthetic Adams spectral sequence.  We
obtain new computational information about $\C$-motivic and classical
stable homotopy groups.
\end{abstract}

\maketitle

\section{Introduction}
\label{sctn:intro}

The topic of this manuscript is the computation of the 
stable homotopy groups 
\[
\pi_n = \colim_k \pi_{n+k}(S^k),
\]
where $\pi_{n+k}(S^k)$ is the group of based homotopy classes
of maps from a sphere $S^{n+k}$ to another sphere $S^k$.
These groups are a central object of study in topology.
For $n > 0$, each stable homotopy group is a finite group \cite{Serre53}.
Therefore, we can compute $\pi_n$ by studying its $p$-primary components
for each prime $p$.  In this manuscript, we consider only the
$2$-primary components.

Building on earlier work of many authors, 
the $2$-primary stable homotopy groups are mostly computed 
for $n \leq 90$ in \cite{IWX} \cite{IWX20a}.  However, there are a few 
gaps in those computations; for example,
$\pi_{82}$ and $\pi_{83}$ are computed only 
up to a factor of $\Z/2$.
Our work completes the computation of
both $\pi_{82}$ and $\pi_{83}$.

\begin{thm}
\label{thm:82-83}
\mbox{}
\begin{enumerate}
\item
The $2$-primary component of $\pi_{82}$ is
$(\Z/2)^6 \oplus \Z/8$.
\item
The $2$-primary component of $\pi_{83}$ is $(\Z/2)^3 \oplus (\Z/8)^2$.
\end{enumerate}
\end{thm}

The proof of Theorem \ref{thm:82-83} is to show that a particular
element $h_6 g + h_2 e_2$ in the Adams spectral sequence is a 
permanent cycle.  
We construct a certain 6-cell complex that allows us to
identify a homotopy class that must be detected by $h_6 g + h_2 e_2$.
The structure of the argument is similar to the proof that the
Kervaire class $h_5^2$ survives \cite{BJM}, but the detailed supporting
computations are different.

From \cite{IWX}*{Lemma 5.24}, we already know that $h_6 g + h_2 e_2$ survives to the
$E_9$-page of the Adams spectral sequence, but there is a possible
non-zero value for $d_9(h_6 g + h_2 e_2)$.

\begin{rmk}
  Mark Behrens suggested an argument using height $2$ chromatic
  homotopy theory that would rule out this differential, which
  would provide
  a completely different argument that $h_6 g + h_2 e_2$ is a 
  permanent cycle.
\end{rmk}

A distinguishing feature of our argument is the use of the
$\F_2$-synthetic Adams spectral sequence.
Motivic stable homotopy theory is useful for classical computations because it is an enhancement of classical stable homotopy theory.  This can be made precise using the language of deformations \cite{GWX17} \cite{GIKR18}  \cite{Pstragowski18}.  
In this manuscript, we use $\F_2$-synthetic stable
homotopy theory, which is a different enhancement of classical stable homotopy theory and provides access to different pieces of information.
Here we are using $\F_2$ as shorthand for the
Eilenberg-Mac Lane spectrum $H\F_2$.

In addition to the study of $h_6 g + h_2 e_2$, this manuscript also
includes the proofs of a number of less prominent computations
about stable homotopy groups.  
These results fill in several gaps left by the motivic techniques
of \cite{IWX}.
Individually, these facts have little direct connection,
but their proofs all use the same basic idea of exploiting the additional structure of $\F_2$-synthetic stable homotopy theory.
Tables \ref{tab:diff} and \ref{tab:extn} summarize the precise
results that we prove.

\begin{thm}
\label{thm:diff-extn}
\mbox{}
\begin{enumerate}
\item
Table \ref{tab:diff} lists some differentials in the $\F_2$-synthetic
Adams spectral sequence.
\item
Table \ref{tab:extn} lists some hidden extensions in the $\F_2$-synthetic
Adams spectral sequence.
\end{enumerate}
\end{thm}

Not all of the results in Tables \ref{tab:diff} and
\ref{tab:extn} are new.  In a few cases, we have provided simpler
proofs of facts that were previously known to be true by more intricate
arguments. 
Very recently announced machine computations \cite{LWX}
independently verify some of our results.  However, 
current machine results do not establish that $h_6 g + h_2 e_2$ 
is a permanent cycle, so our argument regarding that particular
element is the only known proof.

Most of the notation in Tables \ref{tab:diff} and \ref{tab:extn}
is the same as in \cite{IWX}.  The exceptions are $\lambda$,
which is the $\F_2$-synthetic parameter analogous to the motivic parameter
$\tau$; and $\hsf$, which is our notation for an $\F_2$-synthetic homotopy
class such that $\lambda \hsf = 2$.  See Sections 
\ref{subsctn:syn-Adams} and \ref{subsctn:notation}
below for a more careful discussion of notation.

\begin{longtable}{lllll}
\caption{Some Adams differentials \label{tab:diff}
} \\
\toprule
$(s,f,d)$ & element & $r$ & $d_r$ & proof \\
\midrule \endfirsthead
\caption[]{Some Adams differentials} \\
\toprule
$(s,f,d)$ & element & $r$ & $d_r$ & proof \\
\midrule \endhead
\bottomrule \endfoot
$(56, 9, 9)$ & $P h_5 e_0$ & $5$ & $\lambda^4 i l$ & Proposition \ref{prop:d5-Ph5e0} \\
$(61, 6, 6)$ & $A'$ & $5$ & $\lambda^4 M h_1 d_0$ & Proposition \ref{prop:d5-A'} \\
$(70, 4, 4)$ & $p_1 + h_0^2 h_3 h_6$ & $5$ & 
$\lambda^4 h_2^2 C' + \lambda^4 h_3(\D e_1 + C_0)$ & 
\cite{LWX} \\
$(83, 5, 5)$ & $h_6 g + h_2 e_2$ & $9$ & $0$ & Theorem \ref{thm:perm-h6g+h2e2} \\
$(93, 13, 13)$ & $P^2 h_6 d_0$ & $6$ & $0$ & Proposition \ref{prop:perm-P^2h6d0} \\
$(93, 13, 13)$ & $t Q_2$ & $6$ & $\lambda^5 M P \D h_1 d_0$ & Proposition \ref{prop:d6-tQ2} \\
\end{longtable}

\begin{longtable}{lllll}
\caption{Some hidden extensions \label{tab:extn}
} \\
\toprule
$(s,f,d)$ & source & type & target & proof \\
\midrule \endfirsthead
\caption[]{Some hidden extensions} \\
\toprule
$(s,f,d)$ & source & type & target & proof \\
\midrule \endhead
\bottomrule \endfoot
$(71, 5, 5)$ & $h_1 p_1$ & $\nu$ & $0$ & Proposition \ref{prop:nu-h1p1} \\
$(74, 6, 6)$ & $h_3 n_1$ & $\hsf$ & $\lambda x_{74,8}$ & Proposition \ref{prop:h-h3n1} \\
$(77, 7, 7)$ & $m_1$ & $\nu$ & $0$ & Proposition \ref{prop:nu-m1} \\
$(77, 7, 5)$ & $\lambda^2 m_1$ & $\eta$ & $0$ & Proposition \ref{prop:eta-L^2m1} \\
$(77, 12, 12)$ & $M \D h_1 h_3$ & $\nu$ & $\lambda^2 M h_1 e_0^2$ & Proposition \ref{prop:nu-MDh1h3} \\
$(79, 7, 7)$ & $h_2 x_{76,6}$ & $\eta$ & $0$ & Lemma \ref{lem:nu,eta,h2x76,6} \\
$(81, 12, 12)$ & $\D^2 p$ & $\nu$ & $0$ & Proposition \ref{prop:nu-D^2p} \\
$(84, 10, 10)$ & $P x_{76,6}$ & $\hsf$ & $0$ & Proposition \ref{prop:h-Px76,6} \\
\end{longtable}

In order to carry out the computations in Tables \ref{tab:diff} and \ref{tab:extn}, we need a number of technical supporting details,
including some Toda bracket computations as well as some careful
choices of stable homotopy elements.  

\begin{thm}
\label{thm:defn-Toda}
\mbox{}
\begin{enumerate}
\item
There exist elements in the $\F_2$-synthetic stable homotopy groups with the properties shown in Table \ref{tab:elements}.
\item
Table \ref{tab:Toda} lists some Toda brackets in the $\F_2$-synthetic stable homotopy groups.
\end{enumerate}
\end{thm}

\begin{longtable}{llll}
\caption{Some elements in the $\F_2$-synthetic stable
homotopy groups \label{tab:elements}
} \\
\toprule
$(s,d)$ & element & properties & defined in \\
\midrule \endfirsthead
\caption[]{Some elements in the $\F_2$-synthetic stable
homotopy groups} \\
\toprule
$(s,d)$ & element & properties & defined in \\
\midrule \endhead
\bottomrule \endfoot
$(77, 7)$ & $\mu$ & detected by $m_1$ & Lemma \ref{lem:mu} \\
& & $\lambda^2 \eta \mu = 0$\\
& &  $\nu \mu = 0$ \\
$(79, 8)$ & $\beta$ & $\lambda \nu \beta = 0$ & Lemma \ref{lem:beta} \\
& & $\lambda^3 \beta \in \langle \hsf, \lambda^2 \eta, \lambda \mu \rangle$ \\
$(82, 8)$ & $\alpha$ & $2 \alpha = 0$ & Lemma \ref{lem:alpha} \\
& & $\lambda^3 \alpha + \lambda \kappabar \theta_5 \in \langle \lambda^2 \eta, \lambda \mu, \nu \rangle$ \\
\end{longtable}

\begin{longtable}{lllll}
\caption{Some Toda brackets \label{tab:Toda}
} \\
\toprule
$(s,d)$ & bracket & detected by & indeterminacy & proof \\
\midrule \endfirsthead
\caption[]{Some Toda brackets} \\
\toprule
$(s,d)$ & bracket & detected by & indeterminacy & proof \\
\midrule \endhead
\bottomrule \endfoot
$(46, 7)$ & $\langle \eta, \hsf^2 \kappabar_2, \hsf \rangle$ & $M h_1$ & $\lambda^4 d_0 l$ & Lemma \ref{lem:eta,h^2kappabar2,h} \\
$(77, 12)$ & $\langle \hsf^2 \kappabar_2, \hsf, \{\D h_1 h_3\} \rangle$
& $M \D h_1 h_3$ & $0$ & Lemma \ref{lem:h^2kappabar2,h,Dh1h3} \\
$(79, 5)$ & $\langle \lambda \mu, \lambda^2 \eta, \hsf \rangle$ & $0$ or $\lambda^9 M e_0^2$ & $\lambda^3 h_0 h_2 x_{76,6}$ & Lemma \ref{lem:h,L^2eta,Lmu} \\
$(82, 5)$ & $\langle \nu, \lambda \mu, \lambda^2 \eta \rangle$ & $\lambda h_5^2 g$ & ? & Lemma \ref{lem:L^2eta,Lmu,nu} \\
$(84, 8)$ & $\langle \nu, \eta, \{h_2 x_{76,6} \} \rangle$ & $\lambda^2 P x_{76,6}$ & ? &  Lemma \ref{lem:nu,eta,h2x76,6} \\
\end{longtable}

\subsection{The synthetic Adams spectral sequence}
\label{subsctn:syn-Adams}

For simplicity, we use the more concise term
``synthetic" instead of ``$\F_2$-synthetic".  
The only two synthetic homotopy theories under consideration in this
manuscript are the $\F_2$-synthetic and $BP$-synthetic ones.
To avoid confusion, we use the term ``motivic" to describe the latter homotopy theory.

We have made no particular effort to write a self-contained document.  Rather, this manuscript is a companion to \cite{IWX}, and it is just one part of a larger program to compute stable homotopy groups in a range.
We refer to \cite{BHS}*{Section 9 and Appendix A} \cite{Pstragowski18} 
for foundational material on the construction and properties of
synthetic spectra and synthetic Adams spectral sequences.
Also, \cite{I20}*{Chapter 6} foreshadows many of these ideas in the motivic setting.

We provide a short overview of the structure of the synthetic Adams spectral sequence from a computational perspective.
The synthetic Adams $E_2$-page is tri-graded by stem, Adams filtration, and synthetic degree.  The first two gradings are the familiar ones from the classical Adams spectral sequence.

The synthetic homology of a point is equal to
$\F_2[\lambda]$, where the synthetic degree of $\lambda$ is $-1$.
We use the letter $\lambda$ for the synthetic parameter to distinguish it from the motivic parameter $\tau$.  In the long run, we foresee computations that possess two parameters simultaneously, and this notation is convenient for that purpose.

The synthetic Adams $E_2$-page is a free $\F_2[\lambda]$-module whose
generators are in one-to-one correspondence with an $\F_2$-basis for
the classical Adams $E_2$-page.
For each element $x$ of the classical Adams spectral sequence in 
stem $s$ and filtration $f$, there is an $\F_2[\lambda]$-module
generator of the synthetic Adams spectral sequence in stem $s$, filtration $f$, and synthetic degree $f$.  We use the same letter $x$ to represent this generator of the synthetic Adams $E_2$-page.

Different authors have chosen different conventions regarding
the synthetic degree.  
Our choice is based on practical convenience, but unfortunately it is
incompatible with
the traditional use of motivic weight.
More precisely, the Thom reduction map $BP \map H\F_2$ induces a functor
from $BP$-synthetic to $\F_2$-synthetic homotopy theory.
By careful inspection of definitions, this functor takes
$\tau$ to $\lambda^2$.
In terms of stem, motivic weight, and synthetic degree,
the functor induces a homomorphism from the motivic 
homotopy group $\pi_{s,w}$ to the synthetic homotopy group $\pi_{s,2w-s}$.
Beware that the formal suspension has synthetic degree $-1$.  For example,
the cofiber of an element $\gamma$ of the synthetic homotopy group
$\pi_{s,d}$ has two cells in dimensions $(0,0)$ and $(s+1,d-1)$.

The synthetic Adams differentials are in precise correspondence with
classical Adams differentials.  For each classical differential
$d_r(x) = y$, we have a synthetic differential $d_r(x) = \lambda^{r-1} y$.
Consequently, the class $y$ itself is not hit by a differential,
but $y$ is annihilated by $\lambda^{r-1}$ in the $E_\infty$-page.
Moreover, $y$ detects a homotopy class that is annihilated by $\lambda^{r-1}$.

The synthetic Adams spectral sequence can be described as 
the classical Adams spectral sequence ``with history".
The Adams differentials can be reconstructed from the $E_\infty$-page,
since each element annihilated by $\lambda^{r-1}$ corresponds to
a non-zero Adams $d_r$ differential.

If we invert $\lambda$ in the synthetic Adams spectral sequence,
then we recover the classical Adams spectral sequence tensored with
$\F_2[\lambda^{\pm 1}]$.  In more naive terms, the
$\lambda$-free part of the synthetic Adams $E_\infty$-page is
identical to the classical $E_\infty$-page.

On the other hand, the $\lambda$-power torsion elements in the Adams
$E_\infty$-page represent exotic phenomena.  These torsion classes
can be the targets of 
hidden extensions that are not seen classically.
In general, the additional structure encoded in the $\lambda$-power
torsion can be exploited to carry out computations.
These phenomena are well-illustrated in \cite{Burklund21} \cite{BX} \cite{Marek} \cite{BJM24}.

\subsection{$\pi_{70}$ and $\pi_{71}$}

The value of the Adams differential $d_5(p_1 + h_0^2 h_3 h_6)$
given in Table \ref{tab:diff} is not proved in this manuscript.  
Very recently announced machine computations \cite{LWX} led to the
discovery of a mistake regarding this differential.
Note that the value of $d_5(p_1 + h_0^2 h_3 h_6)$ claimed in
\cite{IWX}*{Table 8} is incorrect, 
although \cite{IWX}*{Lemma 5.61} is correct.

We now know that there is a non-zero differential
$d_5(h_1 p_1) = \lambda^4 h_1 h_3 (\D e_1 + C_0)$
that was previously believed to be zero.  
Therefore, the values of $\pi_{70}$ and $\pi_{71}$ in
\cite{IWX}*{Table 1} are incorrect.  The correct values for
the $2$-primary $v_1$-torsion in $\pi_{70}$ and $\pi_{71}$ are
\begin{gather*}
(\Z/2)^6 \oplus \Z/4 \\
(\Z/2)^5 \oplus \Z/4 \oplus \Z/8
\end{gather*}
respectively.

Also, \cite{IWX}*{Table 15} claims that there is a non-zero
hidden $2$-extension from $h_1 h_3 H_1$ to $h_1 h_3 (\D e_1 + C_0)$ 
in the classical 70-stem.  This is incorrect
since $h_1 h_3 (\D e_1 + C_0)$ is now known to be zero
in the classical Adams $E_\infty$-page.

\subsection{Organization}

The main goal of Section \ref{sctn:h6g+h2e2} is to establish
Theorem \ref{thm:perm-h6g+h2e2}, which shows that 
$h_6 g + h_2 e_2$ is a permanent cycle.  This argument depends on
several technical computations in stable homotopy groups, whose
proofs are assembled in Section \ref{sctn:computations}.

Section \ref{sctn:computations} contains a number of computations.  Many of these computations are independent facts whose proofs are not directly connected.  Some of the computations are used in the proofs of later computations.  The results are arranged mostly in order of increasing stem.  However, there are a few exceptions to this general principle 
so that each individual computation depends only on previous computations.

\subsection{Notation}
\label{subsctn:notation}

By default, we work in the synthetic Adams spectral sequence.  Occasionally, we consider the classical or $\C$-motivic Adams spectral sequence; we will be explicit when that occurs.

We use the same notation as \cite{IWX} for elements of the Adams spectral sequence and for elements in stable homotopy groups.
For an element $x$ in the Adams $E_\infty$-page, let
$\{ x\}$ represent the set of all homotopy elements that are detected
by $x$.  This set possesses more than one element when there are
other elements in filtrations higher than the filtration of $x$.

We write $\hsf$ for an element in $\pi_{0,1}$ that 
is detected by $h_0$.  We choose $\hsf$ such that $2 = 1 + 1$
in $\pi_{0,0}$ is equal to $\lambda \hsf$ and is detected by $\lambda h_0$.
Another possible notation for $\hsf$ is $\mathsf{h}$, which 
some authors have used for the motivic homotopy element that is detected
by $h_0$.
It stands for ``hyperbolic" because of its relationship to the hyperbolic plane as an element of the Grothendieck-Witt group.  It also stands for ``Hopf" because it is the zeroth Hopf map (followed by $\eta$, $\nu$, and $\sigma$).

Most of our results are stated with the degrees in which they occur.  
This makes the manuscript
easier to use for readers who are seeking specific results.

\section{The permanent cycle $h_6 g + h_2 e_2$}
\label{sctn:h6g+h2e2}

Our goal is to show that the element $h_6 g + h_2 e_2$ 
in the $83$-stem is a permanent cycle.
We begin by sketching a line of reasoning without supporting details.
Motivated by this sketch,
later in this section we give
a precise argument that relies on a number of specific
computational facts about various stable homotopy groups.
These computational facts appear in the tables in 
Section \ref{sctn:intro}, and their proofs 
are given in Section \ref{sctn:computations}.

One might expect that the element $h_6 g$ detects
the Toda bracket $\langle 2, \theta_5, \kappabar \rangle$
since the relation $2\theta_5 = 0$ (proved in \cite{Xu16}) is suggested by the
Adams differential $d_2(h_6) = h_0 h_5^2$,
and $\kappabar$ is detected by $g$.
There are two problems with this hope.  First,
we do not expect $h_6 g$ itself to be a permanent cycle.
Rather we expect the linear combination
$h_6 g + h_2 e_2$ to survive.  Second,
the Toda bracket is not defined because
$\theta_5 \kappabar$ is non-zero and detected by 
$h_5^2 g$ in the $82$-stem.

But there is a relation $h_5^2 g = h_2 x_1$ in the Adams $E_2$-page.
Therefore, one might hope for a matric Toda bracket of the form
\[
\left\langle
2, 
\left[ \begin{array}{cc} \theta_5 & \xi \end{array} \right],
\left[ \begin{array}{c} \kappabar \\ \nu \end{array} \right]
\right\rangle,
\]
where $\xi$ is detected by $x_1$.  
The existence of such a bracket would require that
$2 \xi = 0$.  This relation is created by the Adams differential
$d_2(e_2) = h_0 x_1$.  This is a promising sign because we anticipate
the appearance of $h_2 e_2$.

But the matric Toda bracket has a fatal flaw
because $x_1$ does not survive; in fact, $d_3(x_1) = h_1 m_1$.
Therefore, we want to replace the entry
$\xi$ in the above matric bracket with something involving
the relation $\eta \mu = 0$, where $\mu$ is detected by $m_1$.

At this point, the language of matric Toda brackets 
of maps between spheres
is no longer
adequate because of varying
lengths of the null compositions that go into the construction.
The cell complex construction that we adopt is a way of turning
the above sketch into a precise argument.

\subsection{The construction of a $6$-cell complex $X$}
\label{subsctn:complex}

In this section, we will construct an explicit cell complex $X$
that has six cells in bidegrees
\[
(0,-2), \ (20,4), \ (3, 1), \ (81, 6), \ (83, 4), \ (83, 7).
\]
We also construct some closely related complexes that
map to $X$ or receive a map from $X$.
Using standard obstruction theory, we will construct $X$ 
by attaching cells to smaller complexes.

The spectrum $X$ can be intuitively described in terms of the
cell diagram in Figure \ref{fig:X}. 
See \cite{BJM} \cite{HLSX} \cite{WX17} for other uses of cell diagrams.
In this figure (and throughout the construction of $X$),
we refer to the elements $\alpha$, $\beta$, and $\mu$
shown in Table \ref{tab:elements}.

\begin{figure}[h!]
\caption{The $6$-cell complex $X$ \label{fig:X}}
$$\xymatrix{
*+[F]{83, 7}  \ar@{-}@/_-55pt/[ddd]^-{\lambda \beta} \\
*+[F]{83, 4} \ar@{-}[d]^-{\lambda^2 \eta} \ar@{-}@/^-35pt/[ddd]_-{\lambda \theta_5} \ar@{-}@/^-65pt/[dddd]_-{\lambda \alpha} \\
*+[F]{81, 6} \ar@{-}[d]^-{\lambda \mu} \\
*+[F]{3, 1} \\
*+[F]{20, 4} \\
*+[F]{0,- 2}
}$$
\end{figure}
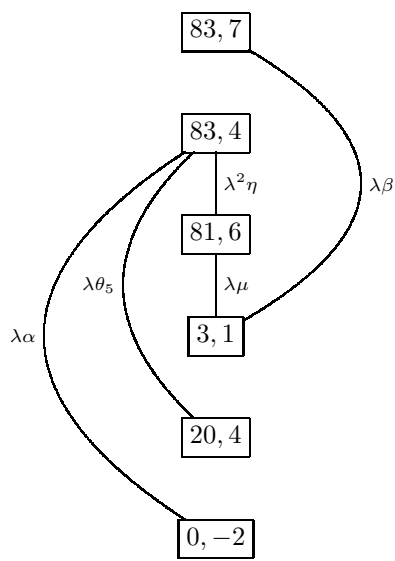

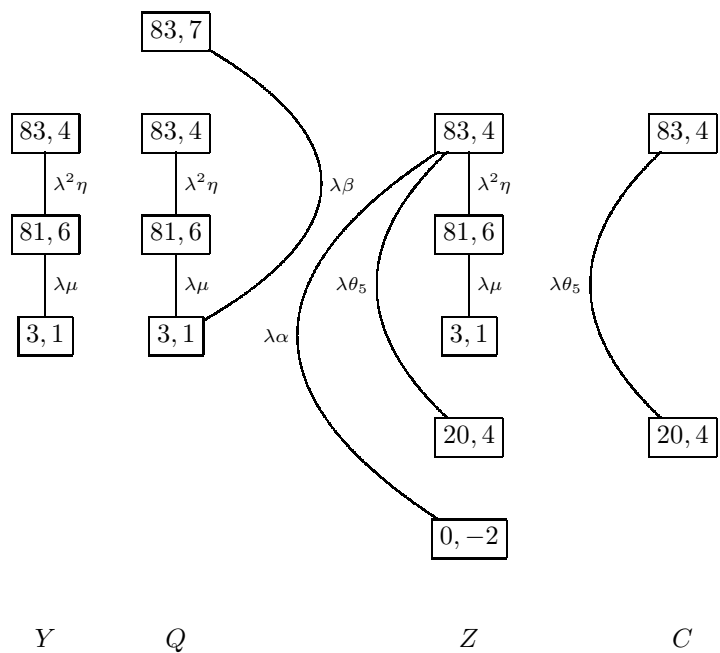
\begin{figure}[h!]
\caption{Some complexes related to $X$ \label{fig:subcomplexes}}
$$\xymatrix{
& *+[F]{83, 7}  \ar@{-}@/_-55pt/[ddd]^-{\lambda \beta} && && \\
*+[F]{83, 4} \ar@{-}[d]^-{\lambda^2 \eta} & *+[F]{83, 4} \ar@{-}[d]^-{\lambda^2 \eta} & && *+[F]{83, 4} \ar@{-}[d]^-{\lambda^2 \eta} \ar@{-}@/^-35pt/[ddd]_-{\lambda \theta_5} \ar@{-}@/^-65pt/[dddd]_-{\lambda \alpha}   && *+[F]{83, 4} \ar@{-}@/^-35pt/[ddd]_-{\lambda \theta_5} \\
*+[F]{81, 6} \ar@{-}[d]^-{\lambda \mu} & *+[F]{81, 6} \ar@{-}[d]^-{\lambda \mu} & && *+[F]{81, 6} \ar@{-}[d]^-{\lambda \mu} \\
*+[F]{3, 1} & *+[F]{3, 1} & && *+[F]{3, 1} \\
& & && *+[F]{20, 4} && *+[F]{20, 4}  \\
& &     && *+[F]{0, -2}  \\
 \Xtwofour & \Xtwofive & && \Xzerofour && \Xonefour
}$$
\end{figure}

\begin{step}
\label{step:X[2,4]}
To begin, Lemma \ref{lem:mu} implies that
$\lambda \mu \cdot \lambda^2 \eta$ is zero.
Therefore, we can 
form a 3-cell complex $\Xtwofour$ with cells in bidegrees 
$(3, 1)$, $(81, 6)$, and $(83, 4)$.
The attaching maps for $\Xtwofour$ are
$\lambda \mu$ and $\lambda^2 \eta$.
The first part of Figure \ref{fig:subcomplexes} shows
a cell diagram for $\Xtwofour$.
Note that $\Xtwofour$ comes equipped with 
inclusion $i_\Xtwofour: S^{3,1} \map \Xtwofour$ of the bottom cell
and projection
$p_\Xtwofour: \Xtwofour \map S^{83,4}$ to the top cell.

Beware that the construction of $\Xtwofour$ depends on a choice
of null-homotopy for $\lambda \mu \cdot \lambda^2 \eta$.
This choice turns out to be immaterial; the later steps of our
construction can be carried out with any fixed null-homotopy.
\end{step}

\begin{step}
\label{step:X[2,5]}
Next, consider the composition 
$i_\Xtwofour \cdot \lambda \beta: S^{82,8} \map \Xtwofour$.
Let $\Xtwofive$ be the cofiber of this map.
The second part of Figure \ref{fig:subcomplexes} shows a cell diagram
for $\Xtwofive$.
Note that $\Xtwofive$ comes equipped with inclusion
$i_\Xtwofive: S^{3,1} \map \Xtwofive$ of the bottom cell.

The top cell of $\Xtwofive$ is only attached to the bottom cell in bidegree $(3,1)$,
so the projection $p_\Xtwofour: \Xtwofour \map S^{83, 4}$ 
extends to a projection $p_\Xtwofive:\Xtwofive \map S^{83, 4}$ to the 
third cell of $\Xtwofive$.
More formally,
the dashed arrow in the commutative diagram 
\[
\xymatrix{
S^{82, 8} \ar[r]^-{i_\Xtwofour \cdot \lambda \beta} & \Xtwofour \ar@{->}[d]^{p_\Xtwofour} \ar[r] & \Xtwofive \ar@{-->}[dl]^-{p_\Xtwofive}\\
 & S^{83, 4} &
}
\]
exists because the top row is a cofiber sequence and because
the composition $p_\Xtwofour \cdot i_\Xtwofour \cdot \lambda \beta$ is zero
(since $p_\Xtwofour \cdot i_\Xtwofour$ is zero).

There is also a projection
$p'_\Xtwofive: \Xtwofive \map S^{83, 7} \vee S^{83,4}$
to the top two cells.  The first component of $p'_\Xtwofive$
is projection to the top cell, and the second component is
the map $p_\Xtwofive$ defined in the previous paragraph.
\end{step}

\begin{step}
\label{step:X[0,5]}
Now consider the composition 
$(\lambda \theta_5, \lambda \alpha) \cdot p_\Xtwofive:
\Xtwofive \map S^{21, 3} \vee S^{1,-3}$.
Define $X$ to be the fiber of this map.

Note that $X$ comes equipped with an inclusion
$i_X: S^{3,1} \vee S^{20,4} \vee S^{0,-2} \map X$
of the bottom three cells because they are not attached to each other.
More formally, consider the diagram
\[
\xymatrix{
& & S^{3,1} \ar@{-->}[dl] \ar[d]^{i_\Xtwofive} \\
S^{20,4} \vee S^{0,-2} \ar[r] & X \ar[r] & \Xtwofive \ar[rr]_-{(\lambda \theta_5, \lambda \alpha) \cdot p_\Xtwofive} & & S^{21,3} \vee S^{1,-3},
}
\]
in which the bottom row is a cofiber sequence.
The last two components of $i_X$ appear in the cofiber
sequence.  The first component, shown as the dashed arrow, exists
because the composition
$(\lambda \theta_5, \lambda \alpha) \cdot p_\Xtwofive \cdot i_\Xtwofive$ is 
zero (since $p_\Xtwofive \cdot i_\Xtwofive$ is zero).
\end{step}

\begin{step}
\label{step:X[0,4]}
We define $\Xzerofour$ to be the fiber of the projection
$p_X: X \map S^{83, 7}$ to the top cell.
The third part of Figure \ref{fig:subcomplexes}
shows a cell diagram for $\Xzerofour$.
Note that $\Xzerofour$ comes equipped with an inclusion
$i_\Xzerofour: S^{3, 1} \vee S^{20, 4} \vee S^{0, -2} \map \Xzerofour$.
More precisely, consider the diagram
\[
\xymatrix{
& S^{3,1} \vee S^{20, 4} \vee S^{0,-2} \ar[d]^{i_X}
\ar@{-->}[dl]_{i_\Xzerofour} \\
\Xzerofour \ar[r] & X \ar[r]_{p_X} & S^{83, 7}
}
\]
in which the bottom row is a cofiber sequence.  The dashed arrow
exists because the composition $p_X i_X$ is zero.
\end{step}

\begin{step}
\label{step:X1,4}
We define $\Xonefour$ to be the cofiber of the map
$\lambda \theta_5: S^{82, 5} \map S^{20,4}$.
The fourth part of Figure \ref{fig:subcomplexes} shows a cell diagram
for $\Xonefour$.
Note that there is an inclusion
$i_\Xonefour: S^{20, 4} \map \Xonefour$ of the bottom cell.

There is a map $\pi: X \map C$ defined as follows.
Consider the diagram
\[
\xymatrix{
\Sigma^{-1} Q \ar[rr]^-{(\lambda \theta_5, \lambda \alpha) \cdot p_Q} &&
S^{20, 4} \vee S^{0, -2} \ar[rr] \ar[d]_{(i_C, 0)} && X \ar@{-->}[dll]^{\pi} \\
&& C,
}
\]
in which the top row is the defining cofiber sequence for
$X$, as described in Step \ref{step:X[0,5]}.
In order to show that the dashed arrow exists, we need only argue that
the composition $(i_C, 0) \cdot (\lambda \theta_5, \lambda \alpha) \cdot
p_Q$ is zero.  This composition equals $i_C \cdot \lambda \theta_5 \cdot p_Q$, and $i_C \cdot \lambda \theta_5$ is zero because
$C$ is the cofiber of $\lambda \theta_5$.
\end{step}

\subsection{The map $f: S^{83,5} \map X$}
\label{subsctn:f}

We next construct a map
$f: S^{83, 5} \map X$, where $X$ is the complex constructed
in Section \ref{subsctn:complex}.
The map $f$ is intuitively described by
the cell diagram in Figure \ref{fig:maps}.
It is represented in the left two columns of the figure.

\begin{figure}[h!]
\caption{The maps $f: S^{83, 5} \map X$ and $g: X \map S^{0,0}$\label{fig:maps}}
$$\xymatrix{
*+[F]{83, 5} \ar[rr]^-{\lambda^2} \ar[rrd]^-{\hsf} && *+[F]{83, 7}  \ar@{-}@/_-55pt/[ddd]^-{\lambda \beta} && \\
 && *+[F]{83, 4} \ar@{-}[d]^-{\lambda^2 \eta} \ar@{-}@/^-35pt/[ddd]_-{\lambda \theta_5} \ar@{-}@/^-65pt/[dddd]_-{\lambda \alpha} &&\\
 && *+[F]{81, 6} \ar@{-}[d]^-{\lambda \mu} &&\\
 && *+[F]{3, 1} \ar[rrdd]^-{\nu} &&\\
 && *+[F]{20, 4} \ar[rrd]^-{\overline{\kappa}} &&\\
 && *+[F]{0, -2} \ar[rr]^-{\lambda^2} && *+[F]{0,0}
}$$
\end{figure}
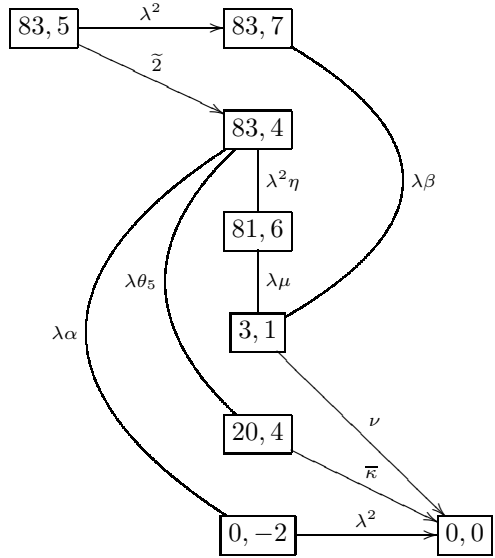

\begin{step}
\label{step:f[2,5]}
Recall from Lemma \ref{lem:beta} that
$\lambda \beta \cdot \lambda^2$ belongs 
to the Toda bracket
$\langle \lambda \mu,  \lambda^2 \eta, \hsf \rangle$.
Therefore,
we have a map $f_\Xtwofive: S^{83,5} \rightarrow \Xtwofive$ 
such that composition
with the projection $p'_\Xtwofive: \Xtwofive \map S^{83,7} \vee S^{83,4}$ is 
$(\lambda^2, \hsf): S^{83,5} \rightarrow S^{83,7} \vee S^{83,4}$.
Recall that the map $p'_\Xtwofive$ was discussed
in Step \ref{step:X[2,5]}.

In fact, one must be careful here about the choice of null-homotopy
in Step \ref{step:X[2,4]}.  We need that $\lambda \beta \cdot \lambda^2$
belongs not merely to the Toda bracket 
$\langle \lambda \mu, \lambda^2 \eta, \hsf \rangle$, but rather to
the subset of this bracket consisting of representatives that are
defined using the specified null-homotopy
of $\lambda \mu \cdot \lambda^2 \eta$.

This turns out to be no problem for us.  We show in
Lemma \ref{lem:beta} that every element of the Toda bracket
$\langle \lambda \mu, \lambda^2 \eta, \hsf \rangle$ is of the form
$\lambda \beta \cdot \lambda^2$.  So we can choose $\beta$
to be compatible with the previously chosen null-homotopy.
\end{step}

\begin{step}
\label{step:f[0,5]}
We construct $f$ as the dashed map in the commutative diagram
\[
\xymatrix{
&& S^{83,5} \ar[d]^{f_\Xtwofive} \ar@{-->}[dll]_{f} \\
X \ar[rr] &&
\Xtwofive \ar[rr]^-{(\lambda \theta_5, \lambda \alpha) \cdot p_\Xtwofive} &&
S^{21,3} \vee S^{1,-3},
}
\]
in which the bottom row is the fiber sequence from Step \ref{step:X[0,5]}.
In order to do this, we just need to argue that the composition
$(\lambda \theta_5, \lambda \alpha) \cdot p_\Xtwofive \cdot f_\Xtwofive$
is zero.
This follows from two observations.
First,
$p_\Xtwofive \cdot f_\Xtwofive: S^{83,5} \map S^{83,4}$ is $\hsf$,
as explained in Step \ref{step:f[2,5]}.
Second $(\lambda \theta_5, \lambda \alpha) \hsf$
is zero by Lemma \ref{lem:alpha} and \cite{Xu16};
here we are using that $2$ equals $\lambda \hsf$.
\end{step} 
 
\subsection{The map $g: X \map S^{0,0}$}
\label{subsctn:g}

We now construct a map
$g: X \map S^{0,0}$, where $X$ is the complex constructed
in Section \ref{subsctn:complex}.
Later in Section \ref{subsctn:composition},
we will study the composition $gf$ as an element of $\pi_{83,5}$.

The map $g$ is intuitively described by
the cell diagram in Figure \ref{fig:maps}.
It is represented in the right two columns of the figure.

\begin{step}
\label{step:g[0,4]}
By Lemma \ref{lem:alpha}, 
$\kappabar \cdot \lambda \theta_5 + \lambda^2 \cdot \lambda \alpha$
is contained in the Toda bracket
$\langle \nu, \lambda \mu, \lambda^2 \eta \rangle$.
Therefore, we can form a map $g_\Xzerofour:\Xzerofour \map S^{0,0}$ 
such that the composition $g_\Xzerofour i_\Xzerofour$
is 
\[
(\nu, \kappabar, \lambda^2): S^{3,1} \vee S^{20,4} \vee S^{0,-2} \rightarrow S^{0,0}.
\]

As in Step \ref{step:f[2,5]},
one must be careful here about the choice of null-homotopy
in Step \ref{step:X[2,4]}.  We need that 
$\kappabar \cdot \lambda \theta_5 + \lambda^2 \cdot \lambda \alpha$
belongs not merely to the Toda bracket 
$\langle \nu, \lambda \mu, \lambda^2 \eta \rangle$, but rather to
the subset of this bracket consisting of representatives that are
defined using the specified null-homotopy
of $\lambda \mu \cdot \lambda^2 \eta$.

This turns out to be no problem for us.  We show in
Lemma \ref{lem:alpha} that every element of the Toda bracket
$\langle \lambda \mu, \lambda^2 \eta, \hsf \rangle$ is of the form
$\kappabar \cdot \lambda \theta_5 + \lambda^2 \cdot \lambda \alpha$.
So we can choose $\alpha$
to be compatible with the previously chosen null-homotopy.
\end{step}

\begin{step}
\label{step:g[0,5]}
Consider the commutative diagram
\[
\xymatrix{
S^{82,8} \ar[r] & \Xzerofour \ar[r] \ar[d]_{g_\Xzerofour} & X \ar[r]^-{p_X} \ar@{-->}[dl]^-{g} & S^{83,7} \\
 & S^{0,0} & &
}
\]
in which the top row is a cofiber sequence.
The composition $S^{82,8} \map \Xzerofour \map S^{0,0}$
is zero because $\nu \cdot \lambda \beta$ is zero.
Therefore, $g_\Xzerofour$ extends to the dashed arrow
$g: X \map S^{0,0}$ shown in the diagram.
\end{step}

\subsection{The composition $g f: S^{83,5} \map S^{0,0}$}
\label{subsctn:composition}

We will detect the map $g f: S^{83,5} \map S^{0,0}$ by smashing with $S^{0,0}/\lambda$ and passing
to the category of $S^{0,0}/\lambda$-modules.  This latter category is
entirely algebraic, i.e., is equivalent to the category of
derived comodules over the classical dual Steenrod algebra.  
In more concrete
terms, the synthetic homotopy groups of $S^{0,0}/\lambda$ are isomorphic
to the classical Adams $E_2$-page, i.e., the cohomology of the dual Steenrod algebra.  These latter homotopy groups of $S^{0,0}/\lambda$
are much more easily understood than the homotopy 
groups of $S^{0,0}$.

We will now show that the complex $X$ splits completely 
after smashing with $S^{0,0}/\lambda$.  This means that maps into and out of
$X/\lambda$ are easy to compute.

\begin{prop}
\label{prop:split/L}
The $S^{0,0}/\lambda$-module $X/\lambda$
splits as a wedge
of suspensions of $S^{0,0}/\lambda$.
\end{prop}

\begin{proof}
The $S^{0,0}/\lambda$-module
$\Xtwofour/\lambda$ has three cells:
$S^{83, 4}/\lambda$, $S^{81, 6}/\lambda$, and $S^{3,1}/\lambda$.
The attaching map between the bottom and middle cells is 
the smash product of $\lambda \mu$ with $S^{0,0}/\lambda$.
This smash product is zero because $\lambda \mu$ is a multiple of $\lambda$.
Similarly, 
the attaching map between the middle and top cells is zero.
The Adams $E_2$-page is zero in stem $79$ and filtration $4$;
therefore, zero is the only possible attaching map between the 
bottom and top cells.
This shows that $\Xtwofour/\lambda$ splits.

As shown in Step \ref{step:X[2,5]}, the spectrum
$\Xtwofive$ is defined by the cofiber sequence
\[
\xymatrix{
S^{82,8} \ar[r]^-{i_\Xtwofour \cdot \lambda \beta} & \Xtwofour \ar[r] & \Xtwofive. 
}
\]
After smashing with $S^{0,0}/\lambda$, this cofiber sequence splits
because $i_\Xtwofour \cdot \lambda \beta$ is a multiple of $\lambda$.
Consequently, $\Xtwofive/\lambda$ splits as
\[
\left( \Xtwofour/\lambda \right) \vee
\left (S^{83,7}/\lambda \right),
\]
and we have already shown that $\Xtwofour/\lambda$ splits.

As shown in Step \ref{step:X[0,5]}, the spectrum
$X$ is defined by the cofiber sequence
\[
\xymatrix{
X \ar[rr] && \Xtwofive \ar[rr]^-{(\lambda \theta_5, \lambda \alpha) \cdot p_\Xtwofive} && S^{21,3} \vee S^{1,-3}. 
}
\]
After smashing with $S^{0,0}/\lambda$, this cofiber sequence splits
because $(\lambda \theta_5, \lambda \alpha) \cdot p_\Xtwofive$ is a multiple
of $\lambda$.  Consequently, $X/\lambda$ splits as
\[
\left( \Xtwofive/\lambda \right) \vee
\left (S^{20,4}/\lambda \right) \vee
\left (S^{0,-2}/\lambda \right),
\]
and we have already shown that $\Xtwofive/\lambda$ splits.
\end{proof}

\begin{rmk}
\label{rmk:split/L}
Similarly to Proposition \ref{prop:split/L} but much easier, the
$S^{0,0}/\lambda$-module $\Xonefour/\lambda$ splits as a wedge of two
suspensions of $S^{0,0}/\lambda$.  Here $\Xonefour$ is the $2$-cell complex
discussed in Step \ref{step:X1,4}.
\end{rmk}

After smashing $f$ with $S^{0,0}/\lambda$, the map
$f/\lambda$ is of the form
\[
\left( a_{0,-2}, a_{0,1}, a_{2,-1}, a_{80, 4}, a_{63, 1}, a_{83, 7}
\right),
\]
where each $a_{s,f}$ is an element of the Adams $E_2$-page in
stem $s$ and filtration $f$.
Similarly, the map $g/\lambda$ is of the form
\[
\left( b_{83,7}, b_{83,4}, b_{81,6}, b_{3,1}, b_{20,4}, b_{0,-2} \right),
\]
where each $b_{s,f}$ is an element of the Adams $E_2$-page 
in stem $s$ and filtration $f$.
The composition $g f /\lambda$ is equal to
\[
b_{83,7} a_{0,-2} +
b_{83,4} a_{0,1} +
b_{81,6} a_{2,-1} +
b_{3,1} a_{80,4} + 
b_{20,4} a_{63,1} +
b_{0,-2} a_{83,7}.
\]
The computation of $g f /\lambda$ is illustrated in 
Figure \ref{fig:gf/L}.  This figure can be interpreted as a cell diagram
in the category of $S^{0,0}/\lambda$-modules.  However, since $X/\lambda$
splits, the cell diagram is trivial.

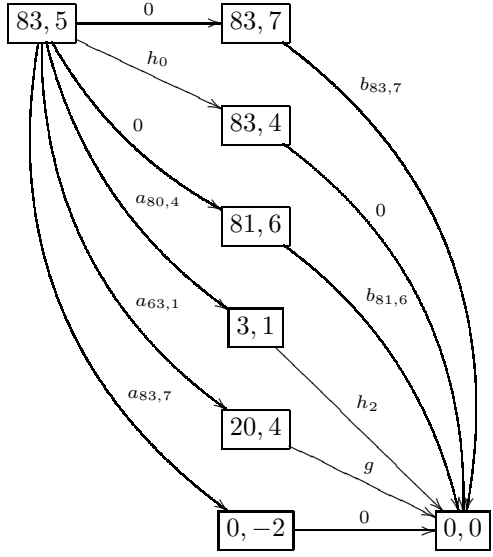
\begin{figure}[h!]
\caption{The composition $gf/\lambda$ \label{fig:gf/L}}
\[
\xymatrix{
*+[F]{83,5} \ar[rr]^-{0} \ar[rrd]^-{h_0} \ar@/^-10pt/[rrdd]^-{0} \ar@/^-15pt/[rrddd]^(.55){a_{80,4}} \ar@/^-25pt/[rrdddd]^(.65){a_{63,1}} \ar@/^-35pt/[rrddddd]^(.7){a_{83,7}} && *+[F]{83,7} \ar@/^35pt/[rrddddd]^(.2){b_{83,7}}  && \\
 && *+[F]{83,4} \ar@/^25pt/[rrdddd]^(.3){0}   &&\\
 && *+[F]{81,6} \ar@/^15pt/[rrddd]^(.35){b_{81,6}}  &&\\
 && *+[F]{3,1} \ar[rrdd]^(.45){h_2} &&\\
 && *+[F]{20,4} \ar[rrd]^-{g} &&\\
 && *+[F]{0,-2} \ar[rr]^-{0} && *+[F]{0,0}
}
\]
\end{figure}

\begin{lemma}
\label{lem:gf/L}
The composition $g f/\lambda: S^{83,5}/\lambda \map S^{0,0}/\lambda$
is equal to $h_2  a_{80,4} + g  a_{63,1}$.
\end{lemma}

\begin{proof}
For degree reasons, $a_{0,2}$, $a_{2,-1}$, $b_{83,4}$, and $b_{0,-2}$
must be zero.
Moreover, $a_{0,1}$ must equal $h_0$ since
$h_0$ detects $\hsf$ and the composition of $f$ with
the projection $X \map S^{83,4}$ is $\hsf$.
Similarly, $b_{20,4}$ must equal $g$ since $g$ detects $\kappabar$
and the composition of $g$ with the inclusion
$S^{20,4} \map X$ is $\kappabar$.
Also, $b_{3,1}$ must equal $h_2$ since $h_2$ detects $\nu$
and the composition of $g$ with the inclusion
$S^{3,1} \map X$ is $\nu$.

Substituting these values, we find that $g f/\lambda$ equals
\[
b_{83,7} \cdot 0 +
0 \cdot h_0 +
b_{81,6} \cdot 0 +
h_2 \cdot a_{80,4} + 
g \cdot a_{63,1} +
0 \cdot a_{83,7},
\]
which simplifies to $h_2 a_{80,4} + g a_{63,1}$.
\end{proof}

We still need to compute that $a_{63,1}$ and $a_{80,4}$ are equal
to $h_6$ and $e_2$ respectively.  We need another cell complex
in order to to accomplish this.

\begin{step}
\label{step:D}
Recall the complex $\Xonefour$ constructed as the cofiber of
$\lambda \theta_5$ in Step \ref{step:X1,4}.
Let $D$ be the cofiber of the composition $\pi f: S^{83,5} \map C$.
A cell diagram for $D$ is shown on the left side of Figure \ref{fig:D}.
\end{step}

\begin{figure}[h!]
\caption{The complexes $D$ and $D/\lambda$ \label{fig:D}}
\[
\xymatrix{
*+[F]{84,4} \ar@{-}[d]^-{\hsf} && *+[F]{84,4} \ar@{-}[d]^-{h_0} \ar@{-}@/_-35pt/[dd]^-{h_6}\\
*+[F]{83,4} \ar@{-}[d]^-{\lambda \theta_5} && *+[F]{83,4} \\
 *+[F]{20,4} &&  *+[F]{20,4} \\
 D && D/\lambda \\
}
\]
\end{figure}
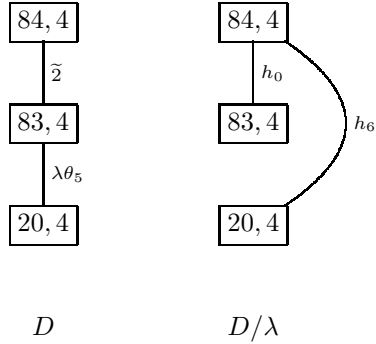

\begin{lemma}
\label{lem:Sq64}
The operation $\Sq^{64}$ is non-zero in the cohomology of $D$.
\end{lemma}

\begin{proof}
We work in the classical context; the synthetic result that we desire
follows immediately by comparison along the realization functor, i.e.,
along $\lambda$-localization.

Recall Adams's decomposition 
$$\Sq^{64} = \Sq^1 \Phi_{5,5} + \cdots$$
of $\Sq^{64}$ 
in terms of secondary operations \cite{Adams60}*{Theorem 4.6.1}.
The secondary operation $\Phi_{5,5}$
is non-zero on the cofiber of $\theta_5$ \cite{BJM}*{p.\ 536}.
\end{proof}

\begin{lemma}
\label{lem:pi-f/L}
The map
$$ \pi f /\lambda: 
S^{83,5}/\lambda \map \Xonefour /\lambda$$
is equal to
$$(h_0, h_6): S^{83,5}/\lambda \map  S^{83,4}/\lambda \vee S^{20,4}/\lambda.$$
\end{lemma}

Here we are using the splitting of $\Xonefour/\lambda$ given in Remark \ref{rmk:split/L}.
Figure \ref{fig:pi-f} illustrates Lemma \ref{lem:pi-f/L} with
cell diagrams.

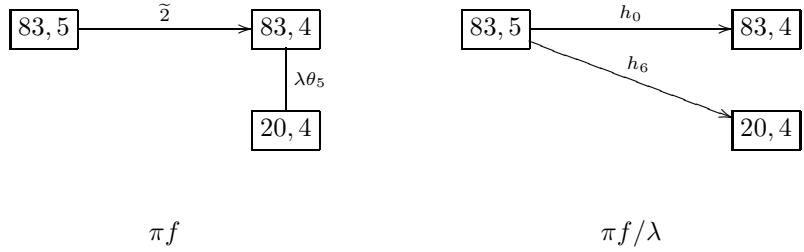
\begin{figure}[h!]
\caption{The maps $\pi f$ and $\pi f / \lambda$ \label{fig:pi-f}}
\[
\xymatrix{
*+[F]{83,5} \ar[rr]^-{\hsf} && *+[F]{83,4} \ar@{-}[d]^-{\lambda \theta_5} && *+[F]{83,5} \ar[rr]^-{h_0} \ar[rrd]^-{h_6}  && *+[F]{83,4} \\
 && *+[F]{20,4} && && *+[F]{20,4} \\
 & \pi f & && & \pi f / \lambda
}
\]
\end{figure}

\begin{proof}
Unlike all of the other $S^{0,0}/\lambda$-modules under consideration,
$D/\lambda$ does not split.  In fact, the middle cell is attached
to the top cell by $h_0$ because $\hsf$ is detected by $h_0$.
Lemma \ref{lem:Sq64} implies that 
$\Sq^{64}$ is non-zero on $D/\lambda$.
Therefore, the bottom cell is attached to the top cell by $h_6$.
A cell diagram for $D/\lambda$ is shown on the right side
of Figure \ref{fig:D}.

The cofiber sequence 
\[
\xymatrix{
S^{83,5} \ar[r]^-{\pi f} & \Xonefour \ar[r] & D
}
\]
that defines $D$ induces a cofiber sequence
\[
\xymatrix{
S^{83,5}/\lambda \ar[r]^-{\pi f/\lambda} & \Xonefour/\lambda \ar[r] & D/\lambda.
}
\]
The computation of the attaching maps for $D/\lambda$ is equivalent
to the computation of $\pi f/\lambda$.
\end{proof}

\begin{lemma}
\label{lem:a63,1}
The value of $a_{63,1}$ is $h_6$.
\end{lemma}

\begin{proof}
The map $\pi/\lambda: X/\lambda \map \Xonefour/\lambda$
is the standard projection onto two wedge summands of
$X/\lambda$.  The computation of
$\pi f/\lambda$ in Lemma \ref{lem:pi-f/L} 
gives that $a_{63,1}$ is equal to $h_6$.
\end{proof}

\begin{thm} \label{thm:perm-h6g+h2e2}
The composition $g f: S^{83,5} \map S^{0,0}$
is detected by $h_6 g + h_2 e_2$ in the synthetic Adams spectral
sequence.
\end{thm}

\begin{proof}
Combining Lemma \ref{lem:gf/L} with Lemma \ref{lem:a63,1},
we get that $g f/\lambda$ is
equal to $h_6 g + h_2 a_{80,4}$ for some value of $a_{80,4}$.
Therefore,
$g f$ is detected by $h_6 g + h_2 a_{80,4}$ in the synthetic
Adams $E_2$-page.

By inspection, $a_{80,4}$ must be a linear combination of 
$e_2$ and $h_1^2 h_4 h_6$.
Since $h_6 g$ itself is known to not
be a permanent cycle \cite{IWX}, 
we must have that $h_2 a_{80,4}$ is non-zero.  The only possibility
is that $h_2 a_{80,4}$ equals $h_2 e_2$.
\end{proof}

\begin{rmk}
The careful reader will note that we did not compute $a_{80,4}$.
It could be either $e_2$ or $e_2 + h_1^2 h_4 h_6$.  In either case,
we conclude that $h_6 g + h_2 e_2$ is a permanent cycle.
\end{rmk}

\section{Computations}
\label{sctn:computations}

\begin{prop}
\label{prop:d5-Ph5e0}
\lemdeg{56, 9, 9}
$d_5(P h_5 e_0) = \lambda^4 i l$.
\end{prop}

\begin{rmk}
Proposition \ref{prop:d5-Ph5e0} immediately follows from \cite{I20}*{Lemma 3.92}, which uses fourfold brackets.  We give a simpler proof.
\end{rmk}

\begin{proof}
Recall the hidden $\eta$ extension from $g^2$ to $\lambda \D h_0^2 e_0$
\cite{IWX}*{Table 18}.
Multiply by $d_0$ to see that there is also a hidden $\eta$ extension
from $e_0^2 g$ to $\lambda i l$.
Also, 
an immediate consequence of the main result of \cite{Burklund21} is that
there is a hidden $\hsf$ extension from
$\lambda^2 h_0 h_5 i$ to $\lambda^4 e_0^2 g$.

Therefore, $\lambda^5 i l$ detects a multiple of $\hsf \eta = 0$,
so it must be hit by a differential.  There is only one possibility.
\end{proof}

\begin{lemma}
\label{lem:eta,h^2kappabar2,h}
\lemdeg{46, 7, 7}
The Toda bracket
$\langle \eta, \hsf^2 \kappabar_2, \hsf \rangle$
is detected by $M h_1$. 
The indeterminacy is generated by 
$\lambda^3 \eta \{ \D h_1 g \}$, which is detected by $\lambda^4 d_0 l$.
\end{lemma}

\begin{rmk}
The Toda bracket in Lemma \ref{lem:eta,h^2kappabar2,h} has a classical
analogue $\langle \eta, 4 \kappabar_2, 2 \rangle$.
The classical bracket contains zero because $\eta \theta_{4.5}$
is detected by $M h_1$.
Synthetically, the product $\eta \theta_{4.5}$ is detected by
$\lambda^3 M h_1$, but $M h_1$ itself does not detect a multiple
of $\eta$.
\end{rmk}

\begin{proof}
Start with
the Massey product $M h_1 = \langle h_1, h_0^2 g_2, h_0 \rangle$
\cite{IWX}*{Table 3}.
Apply the Moss Convergence Theorem \cite{BK21} \cite{Moss70}  to
obtain that 
the Toda bracket
$\langle \eta, \hsf^2 \kappabar_2, \hsf \rangle$ is detected by
$M h_1$.
The indeterminacy is computed by inspection.
\end{proof}

\begin{lemma}
\label{lem:h^2kappabar2,h,Dh1h3}
\lemdeg{77, 12, 12}
The Toda bracket
$\langle \hsf^2 \kappabar_2, \hsf, \{\D h_1 h_3\} \rangle$
is detected by the element $M \D h_1 h_3$, with no indeterminacy.
\end{lemma}

\begin{proof}
The shuffle
\[
\eta \langle \hsf^2 \kappabar_2, \hsf, \{\D h_1 h_3\} \rangle =
\langle \eta, \hsf^2 \kappabar_2, \hsf \rangle \{\D h_1 h_3\}
\]
shows that 
$\eta \langle \hsf^2 \kappabar_2, \hsf, \{\D h_1 h_3\} \rangle$
is detected by $M h_1 \cdot \D h_1 h_3$.  
Here we are using the Toda bracket from Lemma \ref{lem:eta,h^2kappabar2,h}.
It follows that
$\langle \hsf^2 \kappabar_2, \hsf, \{\D h_1 h_3\} \rangle$
is detected by $M \D h_1 h_3$.

The indeterminacy is zero by inspection.
\end{proof}

\begin{prop}
\label{prop:nu-MDh1h3}
\lemdeg{77, 12, 12}
There is a hidden $\nu$ extension from 
$M \D h_1 h_3$ to $\lambda^2 M h_1 e_0^2$.
\end{prop}

\begin{rmk}
\label{rmk:nu-MDh1h3}
The hidden extension in Proposition \ref{prop:nu-MDh1h3}
is purely synthetic; it has no classical (nor motivic) analogue.
In computational terms, the target $\lambda^2 M h_1 e_0^2$ of the extension
is annihilated by $\lambda^2$, so it becomes zero after
$\lambda$-periodicization.

However, we do not know yet that $\lambda^2 M h_1 e_0^2$ is 
annihilated by $\lambda^2$.  That fact is a consequence
of Proposition \ref{prop:d5-A'}, whose proof depends on
Proposition \ref{prop:nu-MDh1h3}.
\end{rmk}

\begin{proof}
Lemma \ref{lem:h^2kappabar2,h,Dh1h3} shows that
the Toda bracket
$\langle \hsf^2 \kappabar_2, \hsf, \{\D h_1 h_3\} \rangle$ is detected
by $M \D h_1 h_3$.
We have
\[
\langle \hsf^2 \kappabar_2, \hsf, \{\D h_1 h_3\} \rangle \nu =
\langle \hsf^2 \kappabar_2, \hsf, \lambda^2 \eta \kappa \kappabar \rangle
\]
because of the hidden $\nu$ extension from $\D h_1 h_3$ to
$\lambda^2 h_1 e_0^2$ \cite{IWX}*{Table 21}.  
Equality holds in the displayed formula because the right side
has no indeterminacy by inspection.

We also have
\[
\langle \hsf^2 \kappabar_2, \hsf, \lambda^2 \eta \kappa \kappabar \rangle =
\langle \hsf^2 \kappabar_2, \hsf, \eta \rangle \lambda^2 \kappa \kappabar,
\]
Because of Lemma \ref{lem:eta,h^2kappabar2,h},
this last expression is detected by $M h_1 \cdot \lambda^2 e_0^2$.
\end{proof}

\begin{prop}
\label{prop:d5-A'}
\lemdeg{61, 6, 6}
$d_5(A') = \lambda^4 M h_1 d_0$.
\end{prop}

\begin{rmk}
Proposition \ref{prop:d5-A'} immediately follows from
\cite{WX17}*{Theorem 12.1}, which requires 8 pages of shuffling several Toda brackets in a delicate way.
We give a simpler proof.
\end{rmk}

\begin{proof}
We will show that $d_5(g A') = \lambda^4 M h_1 e_0^2$, from which the desired
formula follows immediately.

By comparison to motivic homotopy \cite{IWX}*{Table 18}, there is a hidden
$\eta$ extension from $x_{76,9}$ to $\lambda^2 M \D h_1 h_3$.
Proposition \ref{prop:nu-MDh1h3} shows that there is a hidden
$\nu$ extension from $\lambda^2 M \D h_1 h_3$ to
$\lambda^4 M h_1 e_0^2$.  

Therefore, $\lambda^4 M h_1 e_0^2$ must be hit by a differential, but 
there are two possibilities.
By comparison to motivic homotopy,
$d_3(\D^2 p)$ cannot equal $\lambda^2 M h_1 e_0^2$
\cite{IWX}*{Table 6}.
The only remaining possibility is that
$d_5(g A')$ equals $\lambda^4 M h_1 e_0^2$.
\end{proof}

\begin{prop}
\label{prop:nu-h1p1}
\lemdeg{71, 5, 5}
There is no hidden $\nu$ extension on $h_1 p_1$.
\end{prop}

\begin{proof}
The elements
$\lambda P h_0 h_2 h_6$, $\lambda^2 P h_1^3 h_6$, $\lambda^2 x_{74,8}$,
$\lambda^8 \D^2 h_2^2 g$, and $\lambda^{10} e_0^2 g^2$
are the possible values for a hidden $\nu$ extension on $h_1 p_1$.
As shown in \cite{IWX}*{Lemma 7.28}, the element $\tau h_1 p_1$ does not
support a motivic hidden $2$ extension.  Therefore,
$h_1 p_1$ does not support a classical hidden $2$ extension.
Also, $h_1 p_1$ does not support a synthetic hidden $\hsf$ extension
because all of the possible targets are $\lambda$-periodic.
Therefore, the target of a synthetic hidden $\nu$ extension
on $h_1 p_1$ must be annihilated by $\hsf$.
This rules out $\lambda P h_0 h_2 h_6$ and $\lambda^2 x_{74,8}$.

The remaining three cases would imply motivic
hidden $\nu$ extensions from $\tau h_1 p_1$ to
$P h_1^3 h_6$, $\tau \D^2 h_2^2 g$, and $\tau^5 e_0^2 g^2$
respectively.  The first is ruled out by comparison
to $S^{0,0}/\tau$, and the last two are ruled out by comparison to $\mmf$.
\end{proof}

\begin{prop}
\label{prop:h-h3n1}
\lemdeg{74, 6, 6}
There is a hidden $\hsf$ extension from $h_3 n_1$ to $\lambda x_{74,8}$.
\end{prop}

\begin{rmk}
Proposition \ref{prop:h-h3n1} follows from \cite{IWX}*{Lemma 7.35}, which uses fourfold
brackets.  We give a simpler proof.
\end{rmk}

\begin{proof}
Let $\gamma$ be an element of $\pi_{67,5}$ that is detected by 
$Q_3 + n_1$.
Because of the differentials $d_3(d_2) = \lambda^2 h_0 (Q_3 + n_1)$
and $d_4(h_0 d_2) = \lambda^3 X_3$,
projection to the top cell of $S^{0,0}/\lambda$ gives a
hidden $\hsf$ extension from $h_0^2 (Q_3 + n_1)$ to $\lambda X_3$.
Since $h_3 X_3 = h_0^2 x_{74,8}$, we conclude that
$\lambda h_0^2 x_{74,8}$ detects $\hsf^3 \sigma \gamma$.
This implies the desired hidden extension since
$\sigma \gamma$ is detected by $h_3 n_1$.
\end{proof}

\begin{prop} \label{prop:nu-m1}
\lemdeg{77,7,7}
There is no hidden $\nu$ extension on $m_1$.
\end{prop}

\begin{proof}
As shown in \cite{IWX}, the element $m_1$ does not support
a motivic hidden $\nu$ extension, which means that $m_1$ does not
support a classical hidden $\nu$ extension.
Therefore, a $\lambda$-free element cannot be the target of a 
synthetic hidden $\nu$ extension on $m_1$.  By inspection, there are no possible
$\lambda$-torsion targets for a hidden $\nu$ extension on $m_1$.
\end{proof}

\begin{prop} \label{prop:eta-L^2m1}
\lemdeg{77,7,5}
There is no hidden $\eta$ extension on $\lambda^2 m_1$.
\end{prop}

\begin{proof}
The product $h_1 \cdot \lambda^2 m_1$ equals zero in the
synthetic Adams $E_\infty$-page because of the differential
$d_3(x_1) = \lambda^2 h_1 m_1$.

As shown in \cite{IWX}, the element $\tau m_1$ does not support a motivic hidden $\eta$ extension, which means that $m_1$ does not support a classical hidden $\eta$ extension.
Therefore, a $\lambda$-free element cannot be the target of a synthetic 
hidden $\eta$ extension on $\lambda^2 m_1$.
By inspection, there are no possible
$\lambda$-torsion targets for a hidden $\eta$ extension on $\lambda^2 m_1$.
\end{proof}

\begin{lemma} \label{lem:mu}
\lemdeg{77, 7, 7}
There exists an element $\mu$ in $\pi_{77,7}$ that
is detected by $m_1$ such that
$\lambda^2 \eta \mu$ and $\nu \mu$ are both zero.
\end{lemma}

\begin{proof}
Proposition \ref{prop:nu-m1} implies that there exists a choice of
$\mu$ such that $\nu \mu$ is zero.
Moreover, as shown in \cite{IWX}, there are no classical hidden
$\nu$ extensions in higher filtration.
Therefore, there are no synthetic hidden
$\nu$ extensions in higher filtration 
whose targets are $\lambda$-free.
By inspection, there are no $\lambda$-torsion classes that
could be the targets of synthetic hidden $\nu$ extensions 
in higher filtration.  
This implies that $\nu \mu$ is zero for every choice of $\mu$.

Proposition \ref{prop:eta-L^2m1} implies that there exists an element
$\mu'$ in $\pi_{77,5}$ that is detected by $\lambda^2 m_1$
such that $\eta \mu'$ is zero.
We just need to show that $\mu'$ is of the form
$\lambda^2 \mu$, for some choice of $\mu$ detected by $m_1$.

To start, choose $\mu$ arbitrarily.
The sum $\mu' + \lambda^2 \mu$ is detected in higher filtration.
As shown in \cite{IWX}, there are no classical hidden $\eta$
extensions in higher filtration.  Therefore,
there are no synthetic hidden $\eta$ extensions in higher
filtration whose targets are $\lambda$-free.
By inspection, there are no $\lambda$-torsion classes
that could be the targets of synthetic hidden $\eta$
extensions in higher filtration.

However, there is a non-hidden $\eta$ extension
from $\lambda^7 M \D h_1 h_3$ to $\lambda^7 M \D h_1^2 h_3$.
Let $\mu''$ be an element of $\pi_{77,7}$ that is detected
by $\lambda^5 M \D h_1 h_3$, so $\lambda^2 \mu''$ is detected
by $\lambda^7 M \D h_1 h_3$.

If $\mu' + \lambda^2 \mu$ is not equal to $\lambda^2 \mu''$, then
$\eta (\mu' + \lambda^2 \mu)$ is zero.
On the other hand,  if $\mu' + \lambda^2 \mu$ is equal to 
$\lambda^2 \mu''$, then
$\eta (\mu' + \lambda^2 (\mu + \mu''))$ is zero.
Since $\eta \mu'$ was already shown to be zero, we conclude 
that either $\lambda^2 \eta \mu$ or $\lambda^2 \eta (\mu + \mu'')$
is zero.  Consequently, either $\mu$ or $\mu + \mu''$
satisfies the required conditions.
\end{proof}

\begin{lemma} 
\label{lem:h,L^2eta,Lmu}
\lemdeg{79, ?, 5}
The Toda bracket $\langle \lambda \mu, \lambda^2 \eta, \hsf \rangle$ 
contains zero or is detected by $\lambda^9 M e_0^2$, and its indeterminacy
is generated by an element that is detected by
$\lambda^3 h_0 h_2 x_{76,6}$.
\end{lemma}

\begin{rmk}
The Toda bracket
$\langle \lambda \mu, \lambda \eta, \hsf \rangle$
is also defined.
Using the Moss Convergence Theorem \cite{BK21} \cite{Moss70} and the 
differential $d_3(x_1) = \lambda^2 h_1 m_1$,
it is detected by $h_0 x_1$.
The element $h_0 x_1$ is not zero but is annihilated by $\lambda$.
\end{rmk}

\begin{rmk}
The element $\lambda^3 h_0 h_2 x_{76,6}$ has lower Adams filtration
than $\lambda^9 M e_0^2$.
One must often be especially careful in situations like this
when the indeterminacy
is detected in lower filtration because the various elements of the
Toda bracket are detected in different filtrations.
\end{rmk}

\begin{proof}
By inspection, there are no $\lambda$-torsion elements that
could detect the Toda bracket.  Consequently, we only need to
consider $\lambda$-free elements, so we can work in the classical
context.

As shown in \cite{IWX}*{Table 10},
the corresponding motivic Toda bracket
contains zero or is detected by $\tau^2 M e_0^2$.
Therefore, the corresponding classical bracket either contains
zero or is detected by $M e_0^2$.

The indeterminacy is generated by $\lambda \mu \cdot \lambda^3 \eta^2$
and by the multiples of $\hsf$ in $\pi_{79,5}$.
The first expression is zero by Lemma \ref{lem:mu}.
As shown in \cite{IWX}, there are no classical hidden $2$ extensions
in the $79$-stem.  This rules out all possible synthetic hidden $\hsf$
extensions.  There is only one non-hidden $\hsf$ extension
in $\pi_{79,5}$, and its target is $\lambda^3 h_0 h_2 x_{76,6}$.
\end{proof}

\begin{lemma}
\label{lem:beta}
\lemdeg{79, ?, 8}
Every element of $\langle \lambda \mu, \lambda^2 \eta, \hsf \rangle$
is of the form $\lambda^3 \beta$ for some $\beta$ in
$\pi_{79,8}$ such that $\lambda \nu \beta$ is zero.
\end{lemma}

\begin{rmk}
We do not specify which element of the Adams $E_\infty$-page detects
the element $\beta$ in Lemma \ref{lem:beta}.  For our purposes later, 
the detecting element is unimportant.
\end{rmk}

\begin{proof}
The proof involves two steps.  First, we will show that the 
Toda bracket contains an element with the desired properties.
Second, we will show that the elements in the
indeterminacy of the Toda bracket have the desired properties.
These two steps imply that every element of the Toda bracket
has the desired properties.

There are two cases in Lemma \ref{lem:h,L^2eta,Lmu}.
In one case, zero is an element in the Toda bracket with the 
desired properties.

In the other case, let
$\beta$ be detected by $\lambda^6 M e_0^2$.
Then $\lambda^3 \beta$ and 
$\langle \lambda \mu, \lambda^2 \eta, \hsf \rangle$ are both
detected by $\lambda^9 M e_0^2$.
There are no elements in higher filtration, so in fact
$\lambda^3 \beta$ is contained in
$\langle \lambda \mu, \lambda^2 \eta, \hsf \rangle$.
Finally, $\lambda^6 M e_0^2$ cannot support a $\nu$ extension
because there are no possible targets.
This shows that the Toda bracket contains an element with the
desired properties.

Next, we study the elements in the indeterminacy.
Lemma \ref{lem:h,L^2eta,Lmu} shows that the indeterminacy 
is generated by $\lambda^3 \hsf \gamma$, where
$\gamma$ is an element of $\pi_{79,7}$ that is detected by
$h_2 x_{76,6}$.  We want to show that
$\lambda \nu \cdot \hsf \gamma$ is zero.

Note that $\nu \gamma$ is detected by $h_2^2 x_{76,6}$.
We know from \cite{IWX}*{Lemma 7.43} that $h_2^2 x_{76,6}$ does not support a classical
hidden $2$ extension.  Therefore, it also does not support a 
synthetic $\hsf$ extension because all possible targets for such a
hidden extension are $\lambda$-free.
This shows that $\lambda \nu \cdot \hsf \gamma$ is zero.
\end{proof}

\begin{prop}
\label{prop:nu-D^2p}
\lemdeg{81, 12, 12}
There is no hidden $\nu$ extension on $\D^2 p$.
\end{prop}

\begin{proof}
The elements $\lambda \D^2 t$ and $\lambda^2 M \D h_1 d_0$
are the two possible targets for a hidden $\nu$ extension
on $\D^2 p$.  Both elements support $h_1$ extensions,
so they cannot be targets of hidden $\nu$ extensions.
\end{proof}

\begin{lemma}
\label{lem:L^2eta,Lmu,nu}
\lemdeg{82, 6, 5}
Every element of the Toda bracket
$\langle \nu, \lambda \mu, \lambda^2 \eta \rangle$
is detected by $\lambda h_5^2 g$.
Here $\mu$ is the element in $\pi_{77,7}$ specified in
Lemma \ref{lem:mu}.
\end{lemma}

\begin{rmk}
Lemma \ref{lem:L^2eta,Lmu,nu} does not compute the indeterminacy
of the Toda bracket $\langle \nu, \lambda \mu, \lambda^2 \eta \rangle$.
In fact, the bracket does have indeterminacy because of the presence
of multiples of $\lambda^2 \eta$ and of $\nu$ in higher filtration.
However, our later arguments do not depend on the indeterminacy.
\end{rmk}

\begin{proof}
The bracket is well-defined because of
Lemma \ref{lem:mu}.
Using the synthetic Adams differential $d_3(\lambda x_1) = 
\lambda^3 h_1 m_1$,
the Moss Convergence Theorem \cite{BK21} \cite{Moss70} 
implies that $\lambda h_2 x_1 = \lambda h_5^2 g$ detects the Toda bracket.

For degree reasons, all of the indeterminacy is detected in higher
filtration.
Consequently, every element of the bracket is detected by $\lambda h_5^2 g$.
\end{proof}

\begin{lemma}
\label{lem:alpha}
\lemdeg{82, ?, 8}
Every element of the Toda bracket
$\langle \nu, \lambda \mu, \lambda^2 \eta \rangle$ is of the form
$\lambda^3 \alpha + \lambda \kappabar \theta_5$,
where $\alpha$ is an element of $\pi_{82,8}$ such that $2 \alpha$ is zero.
\end{lemma}

\begin{rmk}
We do not specify which element of the Adams $E_\infty$-page detects
the element $\alpha$ in Lemma \ref{lem:alpha}.  
By inspection of the proof, we know that $\alpha$ is detected in
Adams filtration $8$ or higher.
For our purposes later, 
the detecting element is unimportant.
\end{rmk}

\begin{proof}
Because of Lemma \ref{lem:L^2eta,Lmu,nu},
we know that $\lambda h_5^2 g$ detects both
the product $\lambda \kappabar \theta_5$ and every element of
the Toda bracket $\langle \nu, \lambda \mu, \lambda^2 \eta \rangle$.
Therefore, the expression
$$\lambda \kappabar \theta_5 + \langle \nu, \lambda \mu, \lambda^2 \eta \rangle$$
consists entirely of elements that are 
detected in higher filtration.
By inspection of the possible detecting elements, 
every element of
\[
\lambda \kappabar \theta_5 + \langle \nu, \lambda \mu, \lambda^2 \eta \rangle 
\]
is a multiple of $\lambda^3$.
This shows that $\alpha$ exists.

It remains to show that $2 \alpha$ is zero.
By inspection, every possible value
of $2 \alpha$ is $\lambda$-free.  Therefore, it suffices to show
that $2 \lambda^3 \alpha$ is zero.  
We have
\[
2 \left( \lambda \kappabar \theta_5 + \langle \nu, \lambda \mu, \lambda^2 \eta \rangle \right) = 
\langle \nu, \lambda \mu, \lambda^2 \eta \rangle \lambda \hsf 
\]
because $2 \theta_5$ is zero 
(proved in \cite{Xu16})
and because $2 = \lambda \hsf$.
Then shuffle to obtain
\[
\lambda \nu \langle \lambda \mu, \lambda^2 \eta, \hsf \rangle.
\]
Finally, 
in either case of Lemma \ref{lem:h,L^2eta,Lmu}, the expression
$\lambda \nu \langle \lambda \mu, \lambda^2 \eta, \hsf \rangle$
contains zero.
\end{proof}

\begin{lemma}
\label{lem:nu,eta,h2x76,6}
\lemdeg{79, 7, 7}
\lemdeg{84, 10, 8}
There exists an element $\gamma$ in $\pi_{79,7}$ that is detected by
$h_2 x_{76,6}$ such that $\eta \gamma$ is zero.
Moreover,
$\lambda^2 P x_{76,6}$ detects an element in the Toda bracket
$\langle \nu, \eta, \gamma \rangle$.
\end{lemma}

\begin{rmk}
\label{rem:nu,eta,h2x76,6}
Beware that Lemma \ref{lem:nu,eta,h2x76,6} does not compute the
indeterminacy of the Toda bracket.
In fact, the indeterminacy
contains an element that is detected by $\lambda h_2 c_1 H_1$,
whose filtration is lower than the filtration of $\lambda^2 P x_{76,6}$.
One must often be especially careful in situations like this
when the indeterminacy
is detected in lower filtration because the various elements of the
Toda bracket are detected in different filtrations.
\end{rmk}

\begin{proof}
We know from \cite{IWX} that $h_2 x_{76,6}$ does not support a classical
$\eta$ extension.  Moreover, there are no possible targets for a
hidden $\eta$ extension on $h_2 x_{76,6}$ that are $\lambda$-torsion.
Therefore, $h_2 x_{76,6}$ does not support a synthetic $\eta$ extension,
and it is possible to choose $\gamma$.
(The element $\lambda P h_6 c_0$ in higher filtration supports an
$h_1$ multiplication.  This means that not all possible choices of $\gamma$
are annihilated by $\eta$.)

There is a hidden $\C$-motivic
$\nu$ extension from $h_2^2 x_{76,6}$ to $P h_1 x_{76,6}$ \cite{IWX}*{Table 21}.
Therefore, there is also a classical $\nu$ extension, as well
as a synthetic $\nu$ extension from $h_2^2 x_{76,6}$ to
$\lambda^2 P h_1 x_{76,6}$.
Shuffle to obtain
\[
\nu^2 \gamma = \langle \eta, \nu, \eta \rangle \gamma =
\eta \langle \nu, \eta, \gamma \rangle.
\]
Therefore, the right side is also detected by $\lambda^2 P h_1 x_{76,6}$,
and the Toda bracket must be detected by $\lambda^2 P x_{76,6}$.
\end{proof}

\begin{prop}
\label{prop:h-Px76,6}
\lemdeg{84, 10, 10}
There is no hidden $\hsf$ extension on $P x_{76,6}$.
\end{prop}

\begin{proof}
The only possible targets for a hidden $\hsf$ extension are
$\lambda^3 \D^2 t$ and $\lambda^4 M \D h_1 d_0$.
The first possibility is ruled out by \cite{IWX} because the
motivic weights are incompatible.

If there were a hidden $\hsf$ extension 
from $P x_{76,6}$ to $\lambda^4 M \D h_1 d_0$,
then there would also be a hidden $\hsf$ extension from
$\lambda^2 P x_{76,6}$ to $\lambda^6 M \D h_1 d_0$.
The possible differentials hitting $\lambda^8 M \D h_1 d_0$
and $\lambda^9 M \D h_1 d_0$
do not affect this argument.

Therefore, it suffices to show that there is no hidden $\hsf$ extension
on $\lambda^2 P x_{76,6}$.
As in Lemma \ref{lem:nu,eta,h2x76,6}, let $\gamma$
be an element of $\pi_{79,7}$ that is detected by
$h_2 x_{76,6}$ such that $\eta \gamma$ is zero,
so $\lambda^2 P x_{76,6}$ detects $\langle \nu, \eta, \gamma \rangle$.

Consider the relations
\[
\langle \nu, \eta, \gamma \rangle \hsf \subseteq
\langle \nu, \eta, \hsf \gamma \rangle \supseteq
\langle \nu, \eta, \hsf \rangle \gamma.
\]
The last expression is zero because $\langle \nu, \eta, \hsf \rangle$
is zero in $\pi_{5,2}$.
Consequently,
$\langle \nu, \eta, \gamma \rangle \hsf$ lies in the indeterminacy
of the middle expression, which consists entirely of multiples of
$\nu$.

It remains to show that $\lambda^6 M \D h_1 d_0$ cannot detect
a multiple of $\nu$.  This follows from inspection and
Proposition \ref{prop:nu-D^2p}.
\end{proof}

\begin{prop}
\label{prop:perm-P^2h6d0}
\lemdeg{93, 13, 13}
The element $P^2 h_6 d_0$ is a permanent cycle.
\end{prop}

\begin{proof}
We begin by computing
$\langle \hsf, \lambda \theta_5, \{\lambda^3 P^2 d_0\} \rangle$ in $S^{0,0}/\lambda^5$.  
There are no crossing differentials since the
differential $d_6(h_0^3 \cdot \D h_2^2 h_6) = \lambda^5 M \D h_2^2 e_0$
does not occur in $S^{0,0}/\lambda^5$.
Compute the corresponding Massey product in the $E_3$-page using
the differential $d_2(h_6) = \lambda h_0 h_5^2$.
We obtain $\lambda^3 P^2 h_6 d_0$, with no indeterminacy in the $E_3$-page.
Therefore, the Moss Convergence Theorem implies that
the Toda bracket is detected by $\lambda^3 P^2 h_6 d_0$ in
$S^{0,0}/\lambda^5$.

On the other hand, $\{\lambda^3 P^2 d_0\}$ equals zero since
$d_4(d_0 e_0 + h_0^7 h_5) = \lambda^3 P^2 d_0$.
Therefore, the bracket in $S^{0,0}/\lambda^5$ consists entirely of multiples
of $\hsf$.  In particular, $\lambda^3 P^2 h_6 d_0$ detects a
$\hsf$ multiple
in $S^{0,0}/\lambda^5$.  
The only possibility is that there is a hidden $\hsf$ extension from 
$\lambda^2 h_0^4 \cdot \D h_2^2 h_6$ to $\lambda^3 P^2 h_6 d_0$ in $S^{0,0}/\lambda^5$.

Pull back this hidden $h_0$ extension along 
the map $S^{0,0} \map S^{0,0}/\lambda^5$.  We conclude that
$\lambda^2 h_0^4 \cdot \D h_2^2 h_6$ supports a hidden $\hsf$ extension
in $S^{0,0}$.  By compatibility with the extension in $S^{0,0}/\lambda^5$,
the element $\lambda^3 P^2 h_6 d_0$ is the only possible value for the extension in $S^{0,0}$.
In particular, $\lambda^3 P^2 h_6 d_0$ survives.  In turn, this implies
that $P^2 h_6 d_0$ survives.
\end{proof}

\begin{rmk}
The proof of Proposition \ref{prop:perm-P^2h6d0} shows that
there is a motivic (and also classical) $2$ extension from
$h_0^4 \cdot \D h_2^2 h_6$ to $P^2 h_6 d_0$.
Comparison to $S^{0,0}/\tau$ only gives that the target of the $2$ extension
is $P^2 h_6 d_0$ modulo the possible error term
$\tau^2 M^2 h_2$.
\end{rmk}

\begin{prop}
\label{prop:d6-tQ2}
\lemdeg{93, 13, 13}
$d_6(t Q_2) = \lambda^5 M P \D h_1 d_0$.
\end{prop}

\begin{proof}
In the motivic Adams spectral sequence, 
the element $M P \D h_1 d_0$ is hit by a differential \cite{IWX}*{Table 9}.

Adapting this argument to the synthetic context,
we find that $\eta \kappa \gamma \theta_{4.5}$ is detected
by $\lambda^5 M P \D h_1 d_0$, for some $\gamma$ in $\pi_{32,6}$
that is detected by $\D h_1 h_3$.

The element $\lambda \eta \kappa \theta_{4.5}$ is zero
because all elements of $\pi_{60, 7}$ that are detected in sufficiently
high filtration are also detected by $\tmf$.
(Note that $\eta \kappa \theta_{4.5}$ in $\pi_{60,8}$ is non-zero
and detected by $\lambda^3 M h_1 d_0$, but this is irrelevant.)

We now know that $\lambda \eta \kappa \gamma \theta_{4.5}$ is zero
and is detected by $\lambda^6 M P \D h_1 d_0$.
Therefore, 
$\lambda^6 M \D h_1 d_0$ must be hit by a differential.
Proposition \ref{prop:perm-P^2h6d0}
eliminates 
$P^2 h_6 d_0$, $P^2 h_0 h_6 d_0$, and $P^2 h_0^2 h_6 d_0$ as
possible sources for this differential.
The elements $M^2 h_2$ and $\D^2 h_1 g_2$ are
also eliminated because they are products of permanent cycles.
The only remaining possibility is that $d_6(t Q_2)$ equals
$\lambda^5 M P \D h_1 d_0$.
\end{proof}

\bibliographystyle{amsalpha}

\end{document}